\documentclass{mathscan}
\usepackage{tikz,amssymb}

\newcommand{\Gu}{G^{(0)}}

\newcommand{\inv}{^{-1}}
\newcommand{\lsp}{\operatorname{span}}
\newcommand{\clsp}{\overline{\lsp}}
\newcommand{\supp}{\operatorname{supp}}

\newcommand{\Hh}{\mathcal{H}}
\newcommand{\Kk}{\mathcal{K}}
\newcommand{\Ll}{\mathcal{L}}
\newcommand{\Oo}{\mathcal{O}}

\newcommand{\TT}{\mathbb{T}}
\newcommand{\ZZ}{\mathbb{Z}}

\newtheorem{theorem}{Theorem}
\newtheorem{lemma}[theorem]{Lemma}
\newtheorem{corollary}[theorem]{Corollary}

\theoremstyle{definition}

\newtheorem{notation}[theorem]{Notation}
\theoremstyle{remark}

\newtheorem{rmk}[theorem]{Remark}
\newtheorem{rmks}[theorem]{Remarks}

\title{Groupoid algebras as Cuntz-Pimsner algebras}

\author{Adam Rennie}
\author{David Robertson}
\author{Aidan Sims}

\email{renniea@uow.edu.au, droberts@uow.edu.au, asims@uow.edu.au}
\address{School of Mathematics and Applied Statistics\\
University of Wollongong\\
Wollongong\\ NSW\\ 2522\\
AUSTRALIA}

\thanks{This research was supported by the Australian Research Council.}

\begin{document}

\begin{abstract}
We show that if $G$ is a second countable locally compact Hausdorff \'etale groupoid
carrying a suitable cocycle $c:G\to\ZZ$, then the reduced $C^*$-algebra of $G$
can be realised naturally as the Cuntz-Pimsner algebra of a correspondence over the
reduced $C^*$-algebra of the kernel $G_0$ of $c$. If the full and reduced $C^*$-algebras
of $G_0$ coincide, we deduce that the full and reduced $C^*$-algebras of $G$ coincide. We
obtain a six-term exact sequence describing the $K$-theory of $C^*_r(G)$ in terms of that
of $C^*_r(G_0)$.
\end{abstract}

\maketitle

In this short note we provide a sufficient condition for a groupoid $C^*$-algebra to have a
natural realisation as a Cuntz-Pimsner algebra. The advantage to knowing this is the
complementary knowledge one obtains from the two descriptions.

Our main result starts with a second-countable locally compact Hausdorff \'etale
groupoid $G$ with a continuous cocycle into the integers which is unperforated in the
sense that $G$ is generated by $c^{-1}(1)$. We then show that
the reduced groupoid $C^*$-algebra $C^*_r(G)$ is the Cuntz-Pimsner algebra of a
natural $C^*$-correspondence over the reduced $C^*$-algebra $C^*_r(c^{-1}(0))$ of the
kernel of $c$.

We also show that if $C^*(c^{-1}(0))$ and $C^*_r(c^{-1}(0))$ coincide, then $C^*(G)$ and
$C^*_r(G)$ coincide as well. We finish by applying results of Katsura to present a
six-term exact sequence relating the $K$-theory of $C^*_r(G)$ to that of
$C^*_r(c^{-1}(0))$.

\smallskip

\begin{notation}\label{ntn:standing}
For the duration of the paper, we fix a second-countable locally compact Hausdorff
\'etale groupoid $G$ with unit space $G^{(0)}$ and a continuous cocycle $c : G \to \ZZ$;
that is, a map satisfying $c(\gamma_1\gamma_2) = c(\gamma_1) + c(\gamma_2)$ for
composable $\gamma_1$ and $\gamma_2$. We suppose that $c$ is \emph{unperforated} in the
sense that if $c(\gamma) = n > 0$, then there exist composable $\gamma_1, \dots,
\gamma_n$ such that each $c(\gamma_i) = 1$ and $\gamma = \gamma_1 \dots \gamma_n$. For
$n\in\ZZ$ we write $G_n := c^{-1}(n)$. Recall from \cite{Renault1980} that for $u \in
G^{(0)}$, we write $G_u = s^{-1}(u)$ and $G^u = r^{-1}(u)$. The convolution
multiplication of functions is denoted by juxtaposition, as is multiplication in the
$C^*$-completions.
\end{notation}

\begin{rmks}
(i) Observe that $\Gu \subseteq G_0$, and since $c$ is continuous, $G_0$ is a clopen
subgroupoid of $G$.

(ii) If $c$ is strongly surjective in the sense that $c(r^{-1}(u)) = \ZZ$ for all $u \in \Gu$
\cite[Definition~5.3.7]{A-DR}, and $\gamma \in G_n$, then there exists $\alpha \in
r^{-1}(r(\gamma)) \cap G_1$, and then $\gamma = \alpha (\alpha^{-1}\gamma) \in G_1
G_{n-1}$. So an induction shows that if $c$ is strongly surjective then it is unperforated.
\end{rmks}

\begin{lemma}\label{lem:A0=C*r}
Let $A_0 \subseteq C^*_r(G)$ be the completion of $\{f \in C_c(G) : \supp(f) \subseteq
G_0\}$. Then there is an isomorphism $I_0 : C^*_r(G_0) \to A_0$ extending the identity
map on $C_c(G_0)$.
\end{lemma}
\begin{proof}
Fix $u \in \Gu$, let $\pi^0_u$ be the regular representation of $C^*(G_0)$ on $\Hh^0_u :=
\ell^2((G_0)_u)$ and let $\pi_u$ be the regular representation of $C^*(G)$ on $\Hh_u :=
\ell^2(G_u)$. For $a \in C_c(G_0)$, the subspace $\Hh^0_u \subseteq \Hh_u$ is reducing
for the operator $\pi_u(a)$. Let $P_0 : \Hh_u \to \Hh^0_u$ be the orthonormal projection.
We have
\[
\|\pi^0_u(a)\|
    = \|P_0 \pi_u(a) P_0\|
    \le \|\pi_u(a)\|.
\]
Taking the supremum over all $u$, we deduce that $\|a\|_{C^*_r(G_0)} \le
\|a\|_{C^*_r(G)}$. So there is a $C^*$-homomorphism $\pi : A_0 \to C^*_r(G_0)$ extending
the identity map on $C_c(G_0)$.

The restriction map from $C_c(G)$ to $C_c(\Gu)$ extends to a faithful conditional
expectation\footnote[2]{To see this observe that for $f \in C_c(G)$, we have $f(u) =
(\pi_u(f) \delta_u | \delta_u)$, giving $\|f|_{\Gu}\| = \sup_u (\pi_u(f) \delta_u |
\delta_u) \le \sup_u \|\pi_u(f)\| = \|f\|_{C^*_r(G)}$, so restriction induces an
idempotent map $\Psi : C^*_r(G) \to C_0(\Gu)$ of norm one. Hence
\cite[Theorem~II.6.10.2]{Blackadar} implies that $\Psi$ is a conditional expectation.
Given $a \in C^*_r(G) \setminus \{0\}$,
 \cite[Proposition~II.4.2]{Renault1980} shows that $a \in C_0(G)$, so $a(\gamma) \not=
0$ for some $\gamma$.  Then $\Psi(a^*a)(s(\gamma)) = \sum_{\alpha \in
G_{s(\gamma)}}\overline{a}(\alpha)a(\alpha) \ge |a(\gamma)|^2 > 0$, so $\Psi$ is
faithful.}
$\Psi : C^*_r(G) \to C_0(\Gu)$ and a faithful conditional expectation $\Psi_0 :
C^*_r(G_0) \to C_0(\Gu_0) = C_0(\Gu)$. We clearly have $\pi \circ \Psi = \Psi_0 \circ
\pi$, and so a standard argument (see \cite[Lemma~3.13]{SimsWhiteheadEtAl:xx13}) shows
that $\pi$ is faithful and hence isometric. It follows immediately that $I_0 := \pi^{-1}
: C^*_r(G_0) \to A_0$ is an isometric embedding of $C^*_r(G_0)$ in $C^*_r(G)$ as claimed.
\end{proof}

\begin{lemma}\label{lem:X a hilbert module}
Let $X(G)$ be the completion in $C^*_r(G)$ of the subspace $\{f \in C_c(G) : \supp(f)
\subseteq G_1\}$. Then $X(G)$ is a $C^*$-correspondence over $C^*_r(G_0)$ with module
actions $a \cdot \xi = I_0(a)\xi$ and $\xi \cdot a = \xi I_0(a)$ and inner product
$\langle\cdot, \cdot\rangle_{C^*_r(G_0)} : X(G) \times X(G) \to C^*_r(G)$ given by
$\langle \xi, \eta\rangle_{C^*_r(G_0)} = I^{-1}_0(\xi^* \eta)$.
\end{lemma}
\begin{proof}
Since $c$ is a cocycle, we have $G_1^{-1} = G_{-1}$. That $c$ is a cocycle also implies
that if $f \in C_c(G_m)$ and $g \in C_c(G_n)$ then $fg \in C_c(G_{m+n})$ and $f^* \in
C_c(G_{-m})$, and it follows that $C_c(G_1) C_c(G_0), C_c(G_0) C_c(G_1) \subseteq
C_c(G_1)$ and that $C_c(G_1)^* C_c(G_1) \subseteq C_c(G_0)$. The $C^*$-identity and
Lemma~\ref{lem:A0=C*r} show that $\langle\cdot, \cdot \rangle_{C^*_r(G)}$ is isometric in
the sense that $\|\langle \xi, \xi\rangle_{C^*_r(G_0)}\| = \|\xi\|^2_{C^*_r(G)}$, and in
particular is positive definite. Continuity therefore guarantees that the operations
described determine a Hilbert-module structure on $X(G)$. The left action of $C^*_r(G_0)$
on $X(G)$ is adjointable since $\langle a \cdot \xi, \eta\rangle_{C^*_r(G)} =
(a\xi)^*\eta = \xi^*(a^*\eta) = \langle \xi, a^*\cdot \eta\rangle_{C^*_r(G)}$.
\end{proof}

\begin{notation}\label{ntn:XGinc}
We write $I_1$ for the inclusion map $X(G) \hookrightarrow C^*_r(G)$.
\end{notation}

Recall from \cite{Pimsner} that a Toeplitz representation $(\psi, \pi)$ of a
$C^*$-correspondence $X$ over a $C^*$-algebra $A$ in a $C^*$-algebra $B$ consists of a
homomorphism $\pi : A \to B$ and a linear map $\psi : X \to B$ such that $\pi(a)\psi(x) =
\psi(a\cdot x)$, $\psi(x)\pi(a) = \psi(x\cdot a)$ and $\pi(\langle x,y\rangle_A) =
\psi(x)^*\psi(y)$ for all $x,y \in X$ and $a \in A$.

\begin{lemma}\label{lem:toeplitz rep}
The embeddings $I_1 : X(G) \hookrightarrow C^*_r(G)$ and $I_0 : C^*_r(G_0)
\hookrightarrow C^*_r(G)$ form a Toeplitz representation $(I_1, I_0)$ of $X(G)$ in
$C^*_r(G)$.
\end{lemma}
\begin{proof}
For $\xi \in X(G)$ and $a \in C^*_r(G_0)$, we have $I_0(a)I_1(\xi) = a \xi = I_1(a
\cdot\xi)$ by definition of the actions on $X(G)$. For $\xi,\eta \in C_c(G_1)$, we have
\[
I_0(\langle \xi,\eta \rangle_{A_0})
    = I_0(\xi^*  \eta)
    = \xi^*  \eta
    = I_1(\xi)^* I_1(\eta).
\]
So $(I_1, I_0)$ is a Toeplitz representation as claimed.
\end{proof}

We aim to show that $(I_1, I_0)$ is in fact a Cuntz-Pimsner covariant representation, but
we have a little work to do first.

\begin{lemma} \label{lemma:densefunctions}
For each integer $n > 0$, the space
\[
 \lsp\{f_1  f_2  \cdots  f_n : f_i \in C_c(G_1)\text{ for all } i\}
\]
is dense in $C_c(G_n)$ in both the uniform norm and the bimodule norm of Lemma~\ref{lem:X
a hilbert module}, and we have $C_c(G_n)^* = C_c(G_{-n})$.
\end{lemma}
\begin{proof}
Proposition~4.1 of \cite{Renault1980} shows that the reduced norm on $C_c(G)$ dominates
the uniform norm. As discussed on \cite[page~82]{Renault1980}, the reduced norm is
dominated by the full norm which in turn is dominated by the $I$-norm by
\cite[Proposition~1.4]{Renault1980}. By its definition, the $I$-norm agrees with the
uniform norm on functions supported on bisections, forcing equality of all four norms on
such functions. Thus, on functions supported on open bisections $U \subseteq G_1$, the
uniform norm and the bimodule norm on $C_c(G_1)$ agree. Thus it suffices to show that
$\lsp\{f_1 f_2 \cdots  f_n : f_i \in C_c(G_1)\text{ for all } i\}$ is dense in $C_c(G_n)$
in the uniform norm. Fix $x \neq y \in G_n$. Since $G_n = c\inv(n)$ is locally compact,
by the Stone-Weierstrass theorem it suffices to construct functions $f_1, \dots, f_n \in
C_c(G_1)$ such that $(f_1 \cdots  f_n)(x) = 1$ and $(f_1 \cdots  f_n)(y) = 0$.

Since $c$ is unperforated, there exist $x_1, \dots, x_n \in G_1$ such that $x = x_1 \cdots x_n$.
Choose a precompact open bisection $U_n \subseteq G_1$ such that $x_n \in U_n$ and
such that if $s(x) \not= s(y)$ then $s(y) \not\in s(U_n)$; this is possible because $s$
is a local homeomorphism. For each $1 \le i \le n-1$ choose a precompact open bisection
$U_i \subseteq G_1$ with $x_i \in U_i$. Then $U_1 \cdots U_n$ is a precompact open
bisection containing $x$ because multiplication in $G$ is continuous and open. If $s(x) =
s(y)$, then $y \not\in U_1 \cdots U_n$ because the latter is a bisection; and if $s(x)
\not= s(y)$ then $y \not\in U_1 \cdots U_n$ by choice of $U_n$.

By Urysohn's Lemma, for each $i$ there exists $f_i \in C_c(U_i)$ with $f_i(x_i) = 1$. Then
$(f_1  \cdots  f_n)(x) = 1$ by construction, and the convolution formula says that
$$
\supp(f_1  \cdots  f_n) = \supp(f_1)\supp(f_2) \cdots \supp(f_n) \subseteq U_1 \cdots
U_n,
$$
which yields $(f_1  \cdots  f_n)(y) = 0$.

The involution formula $f^*(\gamma) = \overline{f}(\gamma^{-1})$ and
the cocycle property $c(\gamma^{-1})
= -c(\gamma)$ show that $f \in C_c(G_n)$ if and only if $f^* \in C_c(G_{-n})$.
\end{proof}

\begin{corollary} \label{corollary:groupoidcompacts}
There is an injection $\psi : \Kk(X(G)) \to C^*_r(G_0)$ such that $\psi(\theta_{\xi,
\eta}) = \xi  \eta^*$ for all $\xi, \eta \in C_c(G_1)$.
\end{corollary}
\begin{proof}
By the discussion on page~202 of \cite{Pimsner} (see also
\cite[Section~1]{FowlerRaeburn:IUMJ99}), there is a homomorphism $I_1^{(1)} : \Kk(X(G))
\to C^*_r(G)$ such that $I_1^{(1)}(\theta_{\xi, \eta}) = I_1(\xi) I_1(\eta)^* = \xi
\eta^* = I_0(\xi \eta^*)$. Since $I_0$ is isometric, the composition $\psi := I_0^{-1}
\circ I_1^{(1)}$ is a homomorphism satisfying the desired formula, and we need only show
that $\psi$ is injective.

If $T \in \Kk(X(G))$ and $\psi(T) = 0$, then $I_1^{(1)}(T) = 0$ as well. For
$\eta,\xi,\zeta \in X(G)$, we have
\[
I_1^{(1)}(\theta_{\eta, \xi})I_1(\zeta)
    = I_1(\eta) I_1(\xi)^* I_1(\zeta)
    = I_1(\eta \cdot \langle \xi,\zeta\rangle_{A_0})
    = I_1(\theta_{\eta,\xi}(\zeta)),
\]
and linearity and continuity show that for all $S
\in \Kk(X(G))$ and $\zeta \in X(G)$, we have the formula $I_1^{(1)}(S)I_1(\zeta) = I_1(S\zeta)$. In particular, $0 = I_1^{(1)}(T)I_1(\zeta) =
I_1(T\zeta)$ for all $\zeta \in X(G)$. Since $I_1$ is isometric, we deduce that $T\zeta =
0$ for all $\zeta \in C_c(G_1)$, and so $T = 0$. So $\psi$ is injective.
\end{proof}

Let $X$ be a $C^*$-correspondence over a $C^*$-algebra $A$, and write $\phi : A \to
\Ll(X)$ for the homomorphism implementing the left action. Following \cite{K2004}, we say
that a Toeplitz representation $(\psi, \pi)$ of $X$ is \emph{Cuntz-Pimsner covariant} if
the homomorphism $\psi^{(1)}$ of $\Kk(X)$ satisfying $\psi^{(1)}(\theta_{\xi,\eta}) =
\psi(\xi)\psi(\eta)^*$ satisfies $\psi^{(1)} \circ \phi(a) = \pi(a)$ for all $a$ in the
\emph{Katsura ideal} $\phi^{-1}(\Kk(X)) \cap \ker(\phi)^\perp$.

\begin{lemma}\label{lem:Katsura ideal}
Let $\phi = \phi_{X(G)} : C^*_r(G_0) \to \Ll(X(G))$ be the homomorphism implementing the
left action. Then
\begin{enumerate}
\item\label{it:kerphi} $\ker(\phi) = C^*_r(G_0) \cap C_0(\{g \in G_0 : s(g) \not\in
    r(G_1)\})$;
\item\label{it:kerperp} $\ker(\phi)^\perp = C^*_r(G_0) \cap C_0(\{g \in G_0 : s(g)
    \in \overline{r(G_1)}\})$, and $\clsp\{f  g^* : f,g \in C_c(G_1)\} \subseteq
    \ker(\phi)^\perp$ with equality if $r(G_1)$ is closed;
\item\label{it:Katideal} the Katsura ideal $J_{X(G)}$ is
    $$
    J_{X(G)} = \clsp\{f  g^* : f,g \in C_c(G_1)\};
    $$
\item\label{it:CPrep} the Toeplitz representation $(I_1, I_0)$ of
    Lemma~\ref{lem:toeplitz rep} is Cuntz-Pimsner covariant.
\end{enumerate}
\end{lemma}
\begin{proof}
(\ref{it:kerphi}) Choose $a \in C^*_r(G_0)$; by \cite[Proposition~II.4.2]{Renault1980},
$a \in C_0(G)$. Suppose that $a \not\in \ker\phi$. Then there exists $\xi \in X(G)$ such
that $a\xi = \phi(a)\xi$ is nonzero, so there exist composable $\gamma, \gamma'$ with $a(\gamma) \not=
0$ and $\xi(\gamma') \not= 0$. We have $r(\gamma') \in r(G_1)$ by definition of $X(G)$,
and so $a \not\in C_0(\{g \in G_0 : s(g) \not\in r(G_1)\})$.

On the other hand, suppose that $a \not\in C_0(\{g \in G_0 : s(g) \not\in r(G_1)\})$.
Choose $g \in G_0$ such that $s(g) \in r(G_1)$ and $a(g) \not= 0$. Choose $\gamma \in
G_1$ with $r(\gamma) = s(g)$, fix a precompact open bisection $U \subseteq G_1$
containing $\gamma$ and use Urysohn's lemma to choose $\xi \in C_c(U) \subseteq C_c(G_1)$
such that $\xi(\gamma) = 1$. Then $(\phi(a)\xi)(g\gamma) = \sum_{\alpha\beta =
g\gamma}a(\alpha)\xi(\beta)$. Since $\xi$ is supported on the bisection $U$, and since
$\alpha\beta = g\gamma$ implies $s(\beta) = s(\gamma)$, if $\alpha\beta = g\gamma$ and
$a(\alpha)\xi(\beta) \not= 0$, we have $\beta = \gamma$ and then $\alpha = g$ by
cancellation. So $(\phi(a)\xi)(g\gamma) = a(g) \not= 0$ and in particular $a \not\in
\ker\phi$.

(\ref{it:kerperp}) Suppose that $a \in \ker(\phi)^\perp$. Fix $g \in G_0$ with $s(g)
\not\in \overline{r(G_1)}$. Then there is an open $U \subseteq \Gu \setminus r(G_1)$ with
$s(g) \in U$. By Urysohn, there exists $f \in C_c(U)$ with $f(s(g)) = 1$. Part~(1) gives
$f \in \ker(\phi)$ and so $a f = 0$, which gives $0 = (af)(g) = a(g)$. So $a \in C_0(\{g
\in G_0 : s(g) \in r(G_1)\})$. On the other hand, if $a \in C_0(\{g \in G_0 : s(g) \in
\overline{r(G_1)}\})$ and $f \in \ker(\phi)$, then part~(\ref{it:kerphi}) applied to $f$
shows that $s(\{\gamma : a(\gamma) \not= 0\}) \cap r(\{\gamma : f(\gamma) \not= 0\})$ has
empty interior. Hence $\supp(af)$ has empty interior, and since $af$ is continuous, it
follows that $af = 0$ showing that $a \in \ker(\phi)^\perp$.

The containment $\clsp\{f  g^* : f,g \in C_c(G_1)\} \subseteq C^*_r(G_0) \cap C_0(\{g \in
G_0 : s(g) \in r(G_1)\})$ is immediate from the definition of multiplication in $C_c(G)$.
Now suppose that $r(G_1)$ is closed (and hence clopen), and fix $f \in C_c(\{g \in G_0 :
s(g) \in r(G_1)\})$. Using that $G_1$ is open and that $s(\supp(f))$ is compact, we can
find finitely many precompact open bisections $V_i \subseteq G_1$ such that the $r(V_i)$
cover $s(\supp(f))$. Choose a partition of unity $a_i$ on $s(\supp(f))$ subordinate to
the $r(V_i)$, and define functions $b_i$ supported on the $V_i$ by $b_i(\gamma) :=
\sqrt{a_i(r(\gamma))}$ for $\gamma \in V_i$. Then $f = \sum_i (fb_i) b_i^* \in \lsp\{f
g^* : f,g \in C_c(G_1)\}$.

(\ref{it:Katideal}) For $f,g \in C_c(G_1)$ and $\xi \in C_c(G_1) \subseteq X(G)$, we have
$\phi(f  g^*)(\xi) = f  g^*  \xi = \theta_{f,g}(\xi)$. So $\phi(f  g^*) = \theta_{f,g}
\in \Kk(X(G))$. This combined with part~(\ref{it:kerperp}) shows that $\clsp\{f  g^* :
f,g \in C_c(G_1)\} \subseteq J_{X(G)}$. For the reverse containment, let
\[
K := \clsp\{f  g^* : f,g \in C_c(G_1)\}.
\]
Then $K$ is an ideal of $C^*_r(G_0)$ because $C_c(G_0) C_c(G_1), C_c(G_1)C_c(G_0)
\subseteq C_c(G_1)$. For $f,g \in C_c(G_1)$, the operator $\Theta_{f,g} \in \Kk(X(G))$ is
implemented by multiplication by $fg^*$. Since the norm on $X(G)$ is inherited from that
on $G$, we deduce that there is a well-defined isometric map from $K$ to $\Kk(X(G))$
given by $fg^* \mapsto \Theta_{f,g}$. Now suppose that $a \in J_{X(G)}$. Write $\phi(a) =
\lim_n \sum_i \Theta_{f_{n,i}, g_{n,i}}$ where the $f_{n,i}, g_{n,i}$ all belong to
$C_c(G_1)$. Since $fg^* \mapsto \Theta_{f,g}$ is isometric, the sequence $\big(\sum_i
f_{n,i} g_{n,i}^*\big)_n$ converges, say to $a_1 \in K \subseteq C^*_r(G_0)$. Since
$\phi(a) = \lim_n \phi(\sum_i f_{n,i} g_{n,i}^*)$, we deduce that $a - a_1 \in
\ker(\phi)$. Since $a$, and hence $a^*$ belongs to $\ker(\phi)^\perp$, we deduce that
$a^*(a - a_1) = 0$, so $a^*a = a^*a_1$. Hence $a^*a = a^*a_1 \in \clsp\{f  g^* : f,g \in
C_c(G_1)\}$. Using again that $\clsp\{f g^* : f,g \in C_c(G_1)\}$ is an ideal of $A$, we
deduce that $a \in \clsp\{f g^* : f,g \in C_c(G_1)\}$.

(\ref{it:CPrep}) Part~(\ref{it:Katideal}) shows that $\lsp\{f  g^* : f,g \in C_c(G_1)\}$
is dense in the Katsura ideal $J_{X(G)}$. For $f,g \in C_c(G_1)$ we have $I_0(fg^*) = f
 g^*$ because $I_0$ extends the inclusion $C_c(G_0) \subseteq C_c(G)$ by definition, and
we have just seen that $\phi(f g^*) = \theta_{f,g}$, giving $I_1^{(1)}(\phi(f  g^*)) =
I_1^{(1)}(\theta_{f,g}) = I_1(f)I_1(g)^* = f  g^*$ as well.
\end{proof}

\begin{rmk}
In the published version of this paper, we claimed in Lemma~\ref{lem:Katsura
ideal}(\ref{it:kerperp}) that $\ker(\phi)^\perp = C^*_r(G_0) \cap C_0(\{g \in G_0 : s(g)
\in r(G_1)\}) = \clsp\{f  g^* : f,g \in C_c(G_1)\}$, and in Lemma~\ref{lem:Katsura
ideal}(\ref{it:Katideal}) that $\phi_{X(G)}(a) \in \Kk(X(G))$ for all $a \in C^*_r(G_0)$,
but both statements are incorrect in general. The flaw in the proof was the assertion
that $r(G_1)$ is automatically closed.

To see the problem, consider the groupoid $G$ equal as a topological space to
\[
G = ((0,1) \times \ZZ) \cup (\{0,1\} \times \{0\}) \subseteq [0,1] \times \ZZ
\]
with $(x,m)$ and $(y,n)$ composable if and only if $x = y$, and composition given by
$(x,m)(x,n) = (x,m+n)$. This carries an unperforated cocycle $c : G \to \ZZ$ given by
$c(x,m) = m$. We have $G_1 = (0,1) \times \{1\}$ and $G_0 = [0,1] \times \{0\}$. This
gives $X(G) \cong C_0\big((0,1)\big)$ regarded as a Hilbert $C_0([0,1])$-bimodule under
the actions of pointwise multiplcation. Under this identification, we have
$\phi^{-1}(\Kk(C_0((0,1)))) = C_0((0,1)) \subsetneq C([0,1])$, and $\ker(\phi)^\perp =
C([0,1]) \supsetneq C_0((0,1)) = C_0(\{g \in G_0 : s(g) \in r(G_1)\})$.
\end{rmk}

\begin{notation}
Recall that if the pair $(\psi, \pi)$ is a Toeplitz representation of a
$C^*$-correspondence $X$ in a $C^*$-algebra $B$, then for $n \ge 2$ we write $\psi_n$ for
the continuous linear map from $X^{\otimes n}$ to $B$ such that $\psi(x_1 \otimes \cdots
\otimes x_n) = \psi(x_1) \cdots \psi(x_n)$ for all $x_i \in X$.
\end{notation}

\begin{theorem}\label{thm:main}
Suppose that $G$ is a second-countable locally compact Hausdorff \'etale groupoid, and
that $c : G \to \ZZ$ is an unperforated continuous cocycle as in
Notation~\ref{ntn:standing}. Let $G_0 := c^{-1}(0)$ and $G_1 = c^{-1}(1)$. Let $X(G)$ be
the Hilbert-$C^*_r(G_0)$-module completion of $C_c(G_1)$ described in Lemma~\ref{lem:X a
hilbert module}. The inclusion $I_0 : C_c(G_0) \to C_c(G)$ extends to an embedding $I_0 :
C^*_r(G_0) \to C^*_r(G)$ and the inclusion $I_1 : C_c(G_1) \to C_c(G)$ extends to a
linear map $I_1 : X(G) \to C^*_r(G)$. The pair $(I_1, I_0)$ is a Cuntz-Pimsner covariant
representation of $X(G)$, and the integrated form $I_1 \times I_0$ is an isomorphism of
$\Oo_{X(G)}$ onto $C^*_r(G)$.
\end{theorem}
\begin{proof}
Lemma~\ref{lem:A0=C*r} shows that $I_0$ extends to an embedding of $C^*_r(G_0)$. The map
$I_1$ extends to $X(G)$ by definition of the latter (see Lemma~\ref{lem:X a hilbert
module} and Notation~\ref{ntn:XGinc}). Lemma~\ref{lem:toeplitz rep} says that $(I_1,
I_0)$ is a representation of $X(G)$, and Lemma~\ref{lem:Katsura ideal}(\ref{it:CPrep})
says that this representation is Cuntz-Pimsner covariant.

The grading $c : C^*_r(G)\to \ZZ$ induces an action $\beta$ of $\TT$ on $C^*(G)$ such that
\[\textstyle
 \beta_z(a) = z^{c(a)} a\quad\text{ for all $z \in \mathbb{T}$ and $a \in \bigcup_n C_c(G_n)$.}
\]
For each $n$, let $C_c(G_1)^{\odot n}$ denote the dense subspace of $X(G)^{\otimes n}$
spanned by elementary tensors of the form $z_1 \otimes \cdots \otimes z_n$ with each $z_i
\in C_c(G_1)$. Lemma~\ref{lemma:densefunctions} implies that  for $n \ge 1$, the image
$I_n(C_c(G_1)^{\odot n})$ is dense in $C_c(G_n)$. This implies that $I_1 \times I_0$ is
equivariant for the gauge action $\alpha$ on $\Oo_{X(G)}$ and $\beta$. The
gauge-invariant uniqueness theorem \cite[Theorem~6.4]{K2004} implies that $I_1 \times
I_0$ is injective. Moreover, since each $I_n(C_c(G_1)^{\odot n})^*$ is dense in
$C_c(G_n)^*$, which is $C_c(G_{-n})$ by Lemma~\ref{lemma:densefunctions}, $I_1 \times
I_0$ has dense range, and so is surjective since it is a homomorphism between
$C^*$-algebras.
\end{proof}

\begin{corollary}
Suppose that $C^*(G_0) = C^*_r(G_0)$. Then $C^*(G) = C^*_r(G)$ and so
Theorem~\ref{thm:main} describes an isomorphism $\Oo_{X(G)} \cong C^*(G)$.
\end{corollary}
\begin{proof}
We saw in Lemma~\ref{lem:A0=C*r} that the completion $A_0$ of $C_c(G_0)$ in $C^*_r(G)$
coincides with $C^*_r(G_0)$, and therefore, by hypothesis, with $C^*(G_0)$. So
$\|I_0(a)\| = \|a\|_{C^*(G_0)}$ for all $a \in C_c(G_0)$. For $\xi \in C_c(G_1)$, we have
\[
    \|\xi^*  \xi\|_{C^*(G)} \ge \|\xi^*  \xi\|_{C^*_r(G)}.
\]
The function $\xi^*  \xi$ belongs to $C_c(G_0)$ and so $\|\xi^*  \xi\|_{C^*_r(G)} =
\|\xi^*  \xi\|_{C^*_r(G_0)}$ by Lemma~\ref{lem:A0=C*r}, and this last is equal to $\|\xi
^*  \xi\|_{C^*(G_0)}$ by hypothesis. Since the $I$-norm on $C_c(G)$ agrees with the
$I$-norm on $C_c(G_0)$, the canonical inclusion $C_c(G_0) \hookrightarrow C_c(G)$ is
$I$-norm bounded, and so extends to a $C^*$-homomorphism $C^*(G_0) \to C^*(G)$, which
forces $\|\xi^*  \xi\|_{C^*(G_0)} \ge \|\xi^* \xi\|_{C^*(G)}$. So we have equality
throughout, giving
\[
\|\xi^*  \xi\|_{C^*(G)}
    = \|\xi^*  \xi\|_{C^*_r(G_0)}
    = \|\langle I_1(\xi), I_1(\xi)\rangle_{C^*_r(G_0)}\|.
\]
In particular, $\|\xi\|_{C^*(G)} = \|I_1(\xi)\|$. We deduce that $(I_0, I_1)$ extends to
a Cuntz-Pimsner covariant Toeplitz representation $(\tilde{I}_0, \tilde{I}_1)$ of $X(G)$
in $C^*(G)$. Let $q : C^*(G) \to C^*_r(G)$ denote the quotient map. Then $q \circ
(\tilde{I}_1 \times \tilde{I}_0) = I_1 \times I_0$, which is injective by
Theorem~\ref{thm:main}. Since $\tilde{I}_1 \times \tilde{I}_0$ is surjective, we deduce
that $q$ is also injective.
\end{proof}

\begin{rmk}
If $G_0$ is an amenable groupoid (we don't have to specify which flavour because
topological amenability and measurewise amenability coincide for second-countable locally
compact Hausdorff \'etale groupoids \cite[Theorem~3.3.7]{A-DR}), then
\cite[Proposition~9.3]{Spielberg} shows that $G$ is amenable as well, and then the
preceding corollary follows. So the corollary has independent content only if $G_0$ is
not amenable but nevertheless has identical full and reduced norms.
\end{rmk}

Let
\begin{align*}
L_0 = \{a \in M_2(C_c(G)) : {}&a_{11} \in J_{X(G)}, a_{22} \in C_c(G_0),\\
    &a_{12} \in C_c(G_1) \text{ and } a_{21} \in C_c(G_{-1})\}.
\end{align*}
Then the completion of $L_0$ in the norm induced by the reduced norm on $C^*_r(G)$ is
\[
L = \left(\begin{matrix}
    J_{X(G)} & X(G) \\
    X(G)^* & C^*_r(G_0)
\end{matrix}\right).
\]
Part (\ref{it:Katideal}) of Lemma~\ref{lem:Katsura ideal}  shows that $L$ is a
$C^*$-subalgebra of $M_2(C^*_r(G))$, and that the corners $\big(\begin{smallmatrix}
J_{X(G)} & 0\\ 0 & 0\end{smallmatrix}\big)$ and $\big(\begin{smallmatrix} 0 & 0\\ 0 &
C^*_r(G_0)\end{smallmatrix}\big)$ are full. So the inclusion maps $i^{11} : J_{X(G)} \to L$
into the top-left corner and $i^{22} : C^*_r(G_0) \to L$ into the bottom-right corner
determine isomorphisms $i^{11}_* : K_*(J_{X(G)}) \to K_*(L)$ and $i^{22}_* :
K_*(C^*(G_0)) \to K_*(L)$. The composite $[X] := (i^{22}_*)^{-1} \circ i^{11}_* :
K_*(J_{X(G)}) \to K_*(C^*_r(G_0))$ is the map appearing in \cite[Theorem~8.6]{K2004} (see
\cite[Remark~B4]{K2004}). The inclusion map $\iota : J_{X(G)} \hookrightarrow C^*_r(G_0)$ also
induces a homomorphism $\iota_* : K_*(J_{X(G)}) \to K_*(C^*_r(G_0))$, and Theorem~8.6 of
\cite{K2004} gives the following.

\begin{corollary}
There is an exact sequence in $K$-theory as follows:
\[
\begin{tikzpicture}[xscale=1.5, >=stealth]
    \node (tl) at (0,2) {$K_0(J_{X(G)})$};
    \node (tc) at (3,2) {$K_0(C^*_r(G_0))$};
    \node (tr) at (6,2) {$K_0(C^*_r(G))$};
    \node (br) at (6,0) {$K_1(J_{X(G)})$};
    \node (bc) at (3,0) {$K_1(C^*_r(G_0))$};
    \node (bl) at (0,0) {$K_1(C^*_r(G))$};
    \draw[->] (tl)--(tc) node[pos=0.5, above] {\small$\iota_* - [X]$};
    \draw[->] (tc)--(tr) node[pos=0.5, above] {\small$(I_0)_*$};
    \draw[->] (tr)--(br);
    \draw[->] (br)--(bc) node[pos=0.5, above] {\small$\iota_* - [X]$};
    \draw[->] (bc)--(bl) node[pos=0.5, above] {\small$(I_0)_*$};
    \draw[->] (bl)--(tl);
\end{tikzpicture}
\]
\end{corollary}

\end{document}